\documentclass[a4paper,english,reqno]{amsart}

\usepackage{amssymb,amsthm, amsmath,amsfonts}
\usepackage{bbm}
\usepackage{graphicx}
\graphicspath{{figs/}}
\usepackage{subfigure}
\usepackage{microtype}
\usepackage{booktabs}
\usepackage{enumerate}
\usepackage{pgf}
\usepackage[latin1]{inputenc}
\usepackage{verbatim}
\usepackage{cite}
\usepackage{url}
\usepackage{xcolor}

\usepackage[cmtip,all]{xy}
\xyoption{arc}

\usepackage{pgf,tikz}
\usetikzlibrary{arrows,shapes,calc,positioning}
\usepackage{manfnt}
\usepackage{tikz-cd}
\usepackage{url}
\usepackage{microtype}

\usepackage[font=footnotesize]{caption}

\usepackage{xcolor}
\usepackage{hyperref}

\newcommand{\R}{\mathbb{R}}

\renewcommand{\H}{\mathbb{H}}

\newcommand{\N}{\mathbb{N}} 

\newtheorem{theorem}{Theorem}
\newtheorem{lemma}{Lemma}
\newtheorem{proposition}{Proposition}

\newtheorem{property}{Property} 
\newtheorem*{theoremA*}{Main Theorem}
\newtheorem*{theoremB*}{Main Theorem for a Flute}

\theoremstyle{definition}
\newtheorem{definition}[theorem]{Definition}

\newtheorem{remark}{Remark}

\numberwithin{equation}{section}

\begin{document}
	
	\title[Bellis strong stable sets on 
	infinite hyperbolic surfaces]%
	{Bellis strong stable sets on 
		infinite hyperbolic surfaces}
	
	
	%
	

	
	\author[S. Burniol, F. Dal'Bo, S. Herrero]{ Sergi Burniol Clotet, Françoise Dal'Bo, Sergio Herrero Vila}	
	\address{IMERL, CMAT, Universidad de la República, IRL-CNRS-IFUMI 2030, Uruguay}
	\email{sergi.burniol@gmail.com}
	\address{IRMAR, Universit\'e de Rennes, France, IRL-CNRS-IFUMI 2030, Uruguay}
	\email{francoise.dalbo@univ-rennes.fr}
	\address{IRMAR, Universit\'e de Rennes, Rennes, France}
	\email{sergioherrerovila@outlook.com}
	
	\date{01/04/2026}

	\dedicatory{}
	
	\begin{abstract}		
		We provide a corrected proof of a theorem of A. Bellis on strong stable sets in the unit tangent bundle of certain hyperbolic surfaces. The theorem states that, for vectors whose geodesic rays encounter arbitrarily short closed geodesics, the strong stable set in the dynamical sense does not coincide with the associated horocyclic orbit. The proof is based on Bellis' idea of constructing geodesic rays that wind around infinitely many closed geodesics.
	\end{abstract}
	
	\maketitle
	
	%
	\section{Introduction}
	The purpose of this paper is to give a corrected proof of a theorem stated in the PhD thesis of A. Bellis \cite[Théorème E]{Bellis} on strong stable sets of the geodesic flow $g_t$ on the unit tangent bundle $T^1 S$ of a hyperbolic surface $S$. We consider the Sasaki distance $d_{Sa}$ on $T^1 S$ induced by the hyperbolic metric of $S$. The strong stable set $W^{ss}u$ of a vector $u\in T^1S$ is defined as 
	\[
	W^{ss}u = \{ v\in T^1S \, | \, \lim\limits_{t\to +\infty} d_{Sa}(g_tu,g_tv)=0  \}.
	\]
	Let $h_t$ denote the horocyclic flow on $T^1 S$.
	
	The horocyclic orbit $h_{\mathbb{R}} u$ of any vector $u\in T^1S$ is always contained in $W^{ss} u$ \cite{Dalbo}. Bellis proves that if the infimum of the injectivity radius along the geodesic ray $u(\mathbb{R}^+)$ is positive, then $h_{\mathbb{R}} u = W^{ss} u$ \cite[Théorème D]{Bellis}. When $S$ contains arbitrarily small closed geodesics --in particular, $S$ is of infinite type--, Bellis detects vectors $u$ for which $h_{\mathbb{R}} u$ and  $W^{ss} u$ are different.
	
	\begin{theoremA*}\label{thE} \cite[Théorème E]{Bellis}
		Let $S$ be a hyperbolic surface and $u\in T^1S$. If the geodesic ray $u(\mathbb{R}^+)$ meets an infinite sequence of closed geodesics whose lengths tend to zero, then
		\begin{enumerate}
			\item $h_{\mathbb{R}} u\subsetneq W^{ss}u$,
			\item the set $W^{ss}u $ is an uncountable union of horocyclic orbits $h_{\mathbb{R}} v_i$.
		\end{enumerate} 
	\end{theoremA*}
	
	The proof of Bellis contains several errors. In particular, there is an important gap in the proof of Lemma 4.3.1 in \cite{Bellis}. We are only able to prove a weaker statement (Definition \ref{nested_sequence}, Proposition \ref{prop3}). We present a new argument to fill the gap in the proof (Section \ref{end_proof}). When $S$ is a fine flute surface (see \cite[Section 1.4.1]{Bellis}), our argument is not needed, and Bellis' proof can be considered essentially correct.
	
	We also have simplified and corrected the proofs of Proposition 4.1.5 and Lemma 4.3.4 in \cite{Bellis}. We introduce a distance $d_1$ on $T^1S$ involving only the distance on $S$ and avoiding parallel transport, which appears to be a source of errors in the original arguments. We prove in the appendix that $d_1$ and $d_{Sa}$ define the same strong stable sets.
	
	The paper is self-contained. In Section 3, we formalize the notion of winding around a closed geodesic, which is the key tool used by Bellis. Section 4 is devoted to the proof of the Main Theorem.
	
	
	
	\maketitle
		
%
	\section*{}
	\textbf{Acknowledgment :} The third author wishes to express his sincere gratitude to the Department of Mathematics in Montevideo for its warm hospitality during September 2023.


\section{Preliminaries}

We denote by $\mathbb H^2=\{z\in\mathbb{C}:\operatorname{Im} (z)>0\}$ the hyperbolic plane endowed with the hyperbolic distance $d$, and we will denote by $
\partial \mathbb H^2=\mathbb R\cup\{\infty\}
$
 its boundary at infinity.

Recall that the orientation-preserving isometries of $\mathbb H^2$ are precisely the
Möbius transformations
\[
z\longmapsto \frac{az+b}{cz+d},
\qquad
\begin{pmatrix}
	a&b\\ c&d
\end{pmatrix}\in \operatorname{PSL}(2,\mathbb R).
\]

An isometry $\gamma\in \operatorname{PSL}(2,\mathbb R)$ is called \emph{hyperbolic} if it has two fixed
points on $\partial\mathbb H^2$, denoted by $\gamma^+$ and $\gamma^-$, called respectively
the attractive and repulsive fixed points of $\gamma$. Its translation length along the geodesic $(\gamma^-,\gamma^+)$ is denoted
by $\ell(\gamma)$.

\begin{definition}
	Let $\xi\in\partial\mathbb H^2$ and let $(\sigma(t))_{t\ge 0}$ be a geodesic ray
	converging to $\xi$. The \emph{Busemann cocycle based at $\xi$} is the function
	\[
	B_\xi(x,y):=\lim\limits_{t\to+\infty}\bigl(d(x,\sigma(t))-d(y,\sigma(t))\bigr),
	\qquad x,y\in\mathbb H^2.
	\]
\end{definition}

The Busemann cocycle satisfies the cocycle relation
\[
B_\xi(x,z)=B_\xi(x,y)+B_\xi(y,z)
\qquad \text{for all } x,y,z\in \mathbb H^2.
\]

\begin{definition}
	Fix $\xi\in\partial\mathbb H^2$ and $x_0\in\mathbb H^2$. A \emph{horocycle centered at $\xi$} passing through $x_0$
	is a level set of the map
	\[
	x\longmapsto B_\xi(x,x_0).
	\]
	Equivalently, in the upper half--plane model, horocycles are Euclidean circles tangent
	to $\partial\mathbb H^2=\mathbb R\cup\{\infty\}$ together with horizontal lines.
\end{definition}

We denote by $T^1\mathbb H^2$ the unit tangent bundle of $\mathbb H^2$.
For $\tilde v\in T^1\mathbb H^2$, we write $\tilde v(t)$ for the point at time $t$
along the geodesic determined by $\tilde v$.


The \emph{horocyclic flow} on $T^1\mathbb H^2$ is the one-parameter family
\[
(h_s)_{s\in\mathbb R}:T^1\mathbb H^2\to T^1\mathbb H^2,
\]
where $h_s(\tilde v)$ is obtained by moving $\tilde v$ along the horocycle passing through its basepoint $\tilde v(0)$ and centered at 
its forward endpoint $\tilde v(+\infty)$, with parameter $s$ equal to signed
hyperbolic arc-length. The \emph{geodesic flow} on $T^1\mathbb H^2$ will be denoted by $
(g_t)_{t\in\mathbb R}:T^1\mathbb H^2\to T^1\mathbb H^2. 
$

Although there is a natural distance, called \emph{Sasaki distance} (see \cite{Sasaki}) on $T^1\H^2$, we will consider the following distance \cite{Ballmann}: for $\tilde v,\tilde w\in T^1\mathbb H^2$, define
\[
d_1(\tilde v,\tilde w):=d(\tilde v(0),\tilde w(0))+d(\tilde v(1),\tilde w(1)).
\]

From now on, we fix a torsion-free discrete subgroup
$
\Gamma<\operatorname{PSL}(2,\mathbb R),
$
and we consider the hyperbolic surface
$
S=\Gamma\backslash \mathbb H^2
$
and its unit tangent bundle
$
T^1S=\Gamma\backslash T^1\mathbb H^2.
$

We use a tilde to denote lifts to the universal cover: if $u\in T^1S$, then
$\tilde u\in T^1\mathbb H^2$ denotes any lift of $u$.

Since the geodesic flow and the horocyclic flow on $T^1\mathbb H^2$ commute with
the action of $\Gamma$, they descend to flows on $T^1S$, still denoted by
$
(g_t)_{t\in\mathbb R}
$ \text{and} $
(h_s)_{s\in\mathbb R}.
$
These are respectively called the \emph{geodesic flow} and the \emph{horocyclic flow}
on $T^1S$.

The distance $d_1$ in $T^1\H^2$ induces a distance on $T^1S$ that we will also denote by $d_1$.

\begin{definition}
	Let $u\in T^1S$. The \emph{strong-stable set} of $u$ is
	\[
	W^{ss}u:=\left\{v\in T^1S:\lim\limits_{t\to+\infty} d_1(g_t u,g_t v)=0\right\}.
	\]
\end{definition}

\begin{remark}
 In Appendix \ref{apendice}  we prove that the distances $d_1$ and $d_{S_A}$ define the same strong-stable set. 
\end{remark}


Throughout this paper, $\tilde u_0$ in $T^1\H^2$ is defined by $\tilde u_0(0)=i$, $\tilde u_0(+\infty)=\infty$.
	
	\begin{definition}\label{nested_sequence}
%
Let $(\gamma_n)_{n\geq 0}$ be a sequence of hyperbolic isometries. We say that  $(\gamma_n)_{n\geq 0}$ is a nested sequence if:
\begin{enumerate}
	\item The group $F=\langle \gamma_n \rangle$ is discrete,
	\item  $(\gamma_n^+)_{n\geq 0}$ is an increasing sequence of $\R^+$,
	\item  $(\gamma_n^-)_{n\geq 0}$ is a decreasing sequence of $\R^-$,
	\item $\forall n\geq 0$, $(\gamma_n^-,\gamma_n^+)\cap \Tilde u_0(\R^+)=\{ \tilde u_0(t_n) \}$ with $t_n>0$,
	\item $(\ell(\gamma_n))_{n\geq 0}$ is decreasing and convergent to $0$.
\end{enumerate}
	\end{definition}
	
	\begin{property}
		Let $(\gamma_n)_{n\geq 0}$ be a nested sequence. Then
		\begin{enumerate}
				\item $\lim\limits_{n\to+\infty}\gamma_n^+=\lim\limits_{n\to+\infty}\gamma_n^-=\infty$,
			\item $\lim\limits_{n\to \infty} t_n=+\infty$.
		
		\end{enumerate}
	\end{property}
	
	\begin{proof}

Let us see that
\[
\lim\limits_{n\to+\infty}\gamma_n^+=\lim\limits_{n\to+\infty}\gamma_n^-=\infty.
\]

Suppose that this is not the case. We may assume that
\[
\lim\limits_{n\to+\infty}\gamma_n^+=\xi\in\mathbb R^+\cup\{\infty\},
\qquad
\lim\limits_{n\to+\infty}\gamma_n^-=\eta\in\mathbb R^-\cup\{\infty\}.
\]

\medskip

\noindent
If $\xi\neq \infty$, or $\eta\neq \infty$,
then $\eta\neq \xi$, and the geodesics $(\gamma_n^-,\gamma_n^+)$ converge to the geodesic $(\eta,\xi)$.
Hence $i$ stays at bounded distance from $(\gamma_n^-,\gamma_n^+)$.
Since $\ell(\gamma_n)\to 0$, it follows that
$
d(\gamma_n i,i)
$ is bounded uniformly in $n$, which contradicts the discreteness of $F$.

Thus, $\eta=\xi=\infty$, and conclude that $\lim\limits_{n\to \infty} t_n=+\infty$ from the definition of $t_n$.
	\end{proof}

\section{Winding around a closed geodesic}

\begin{definition}
	Let $\tilde{u}\in T^{1}\mathbb{H}^{2}$ and $\gamma$ a hyperbolic isometry such that
	\[
	\tilde{u}(\mathbb{R}^{+})\cap (\gamma^{-},\gamma^{+})\neq \varnothing.
	\]
	We define $\mathrm{Wind}_{\gamma}(\tilde{u})\in T^{1}\mathbb{H}^{2}$ by:
	\begin{itemize}
		\item $\mathrm{Wind}_{\gamma}(\tilde {u})(0)=\tilde{u}(0)$,
		\item $\mathrm{Wind}_{\gamma}(\tilde{u})(+\infty)=\gamma\bigl(\tilde{u}(+\infty)\bigr)$.
	\end{itemize}
\end{definition}

	Geometrically, when $\gamma$ belongs to a Fuchsian group $\Gamma$, on the surface  $S=\Gamma \backslash\H^2$, the projection of the ray $\mathrm{Wind}_{\gamma}(\tilde{u})(\mathbb{R}^{+})$, denoted by $\mathrm{Wind}_{\gamma}(u)(\mathbb{R}^{+})$, 
is obtained from ${u}(\mathbb{R}^{+})$ after winding around the closed geodesic
corresponding to $(\gamma^{-},\gamma^{+})$. 
 
	\begin{figure}[htbp]
	\centering
	\begin{normalsize} 
		\def\svgwidth{13cm}
		\input{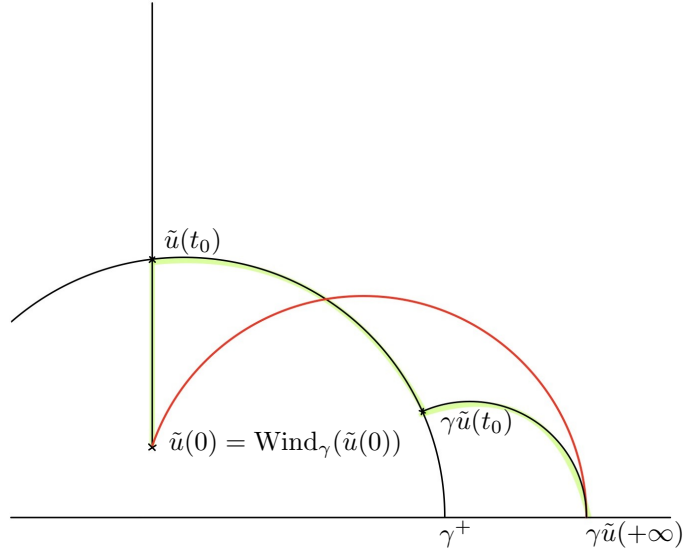}
	\end{normalsize}
	\caption{Winding around a closed geodesic in the universal covering, where $t_0$ the time such that $\tilde u(t_0)\in (\gamma^-,\gamma^+)$.}
	\label{dessin12}
\end{figure}

\begin{definition}[winding time]\label{def_winding_time}

The winding time associated to $\operatorname{Wind}_\gamma(\tilde u)$ is the real number $\tau_{\gamma,\tilde u}$ defined by
$$
\tau_{\gamma,\tilde u}=B_{\tilde u(+\infty)}(\gamma^{-1}\tilde u(0), \tilde u(0)).
$$

\end{definition}

We clearly have $g_{\tau_{\gamma,\tilde u}}(\operatorname{Wind}_\gamma(\tilde u))\in h_{\R}( \gamma \tilde u)$.

 \begin{proposition}[bound of the winding time]\label{bound_wind}
 	Let $\tilde{u}\in T^{1}\mathbb{H}^{2}$ and $\gamma$ a hyperbolic isometry such that the geodesic ray $\tilde u(\R^+)$ meets the axis $(\gamma^-,\gamma^+)$, 
 	then
 	\[
 	\left| \tau_{\gamma,\tilde u} \right| \leq \ell(\gamma).
 	\]
 	
 \end{proposition}
 \begin{proof}
 	Let $p$ be the intersection point between $(\gamma^-,\gamma^+)$ and $\tilde u(\R^+)$.
 	As $p$ belongs to $\tilde u(\R^+)$, we have
 	$$
 	B_{\tilde u(+\infty)}(p,\tilde u(0))=-d(p,\tilde u(0)).
 	$$
 	In addition, as $p\in (\gamma^-,\gamma^+)$ we know that $d(\gamma^{-1}p,p)=\ell(\gamma)$. Thus
 	\begin{align*}
 		B_{\tilde u(+\infty)}(\gamma^{-1}\tilde u(0),\tilde u(0)) 
 		& =  B_{\tilde u(+\infty)}(\gamma^{-1}\tilde u(0),\gamma^{-1}p)+  
 		\\
 		&+B_{\tilde u(+\infty)}(\gamma^{-1}p,p) +B_{\tilde u(+\infty)}(p,\tilde u(0)) \\
 		& \leq  d(\gamma^{-1}\tilde u(0),\gamma^{-1}p)+\ell(\gamma)-d(p,\tilde u(0))\\
 		& = \ell(\gamma).
 	\end{align*}
 	
 	For the other inequality, we take $p_0$ in the intersection $[\gamma^{-1}\tilde u(0),\tilde u(+\infty) )\cap (\gamma^-,\gamma^+)$. This intersection is nonempty, since the isometry $\gamma^{-1}$ preserves the two connected components of $\H^2\setminus(\gamma^-,\gamma^+)$; hence $\gamma^{-1}\tilde u(0)$ lies in the same connected component as $\tilde u(0)$.
 	
 	Then $B_{\tilde u(+\infty)}(\gamma^{-1}\tilde u(0),p_0)=d(\gamma^{-1}\tilde u(0),p_0)$ and $d(p_0,\gamma p_0)=\ell(\gamma)$. Therefore 	
 	\begin{align*}
 		B_{\tilde u(+\infty)}(\gamma^{-1}\tilde u(0),\tilde u(0)) &= B_{\tilde u(+\infty)}(\gamma^{-1}\tilde u(0),p_0)+\\
 		&+B_{\tilde u(+\infty)}(p_0,\gamma p_0)+B_{\tilde u(+\infty) }(\gamma p_0,\tilde u(0)) \\
 		& \geq d(\gamma^{-1}\tilde u(0),p_0)-d(p_0,\gamma p_0)-d(\gamma p_0,\tilde u(0))\\
 		&= - \ell (\gamma).
 	\end{align*}
 \end{proof}
 
 The following proposition will be key for the proof of the main theorem.
 
 
 \begin{proposition}[Key Proposition]\label{key_prop} 
 	Let $\tilde{u}\in T^{1}\mathbb{H}^{2}$ and $\gamma$ a hyperbolic isometry such that the geodesic ray $\tilde u(\R^+)$ meets the axis $(\gamma^-,\gamma^+)$. Then
 	\[\forall\, t\geq0, ~
 	\min \{ d_1\!\left(g_{t+\tau_{\gamma,\tilde u}}\bigl(\operatorname{Wind}_{\gamma}(\tilde u)\bigr),\, g_t\tilde u\right),
 	d_1\!\left(g_{t+\tau_{\gamma,\tilde u}}\bigl(\operatorname{Wind}_{\gamma}(\tilde u)\bigr),\, \gamma g_t\tilde u\right)
 	\}
 	\le 8 \ell(\gamma).
 	\]
 \end{proposition}
 \begin{proof}
 	For simplicity, we denote $\tilde v=\mathrm{Wind}_{\gamma}(\tilde u)$ and $\tau=\tau_{\gamma,\tilde u}$.
 	Up to applying an isometry of $\mathbb{H}^2$, we can assume that $\tilde u(0)=i$, $\tilde u(+\infty)=\infty$ and $\gamma^{+}>0$.
 	
 	Let $t_\gamma\ge 0$ be such that $\tilde u(t_\gamma)\in(\gamma^{-},\gamma^{+})$. Since $i$ and $\gamma \, \infty$ lie in different half-planes bounded by $(\gamma^{-},\gamma^{+})$, the geodesic ray $\tilde v(\mathbb R^{+})$ must intersect the axis $(\gamma^{-},\gamma^{+})$. Let $t'_\gamma\ge 0$ be such that $\tilde v(t'_\gamma)\in(\gamma^{-},\gamma^{+})$.
 	
 	We show that the intersection point $\tilde v(t'_\gamma)\in(\gamma^{-},\gamma^{+})$ lies in the following hyperbolic segment:
 	\[\tilde v(t'_\gamma)\in
 	[\tilde{u}(t_\gamma), \gamma \tilde{u}(t_\gamma)].
 	\]
 	We observe that $\tilde v(+\infty)$ is the real interval $]\gamma^+, +\infty[$ and $\tilde v(-\infty)$ is in the interval $]\gamma^-,0[$. Hence, $\tilde v(t'_\gamma)$ is on the right of $\tilde u(t_\gamma)$.  Applying $\gamma^{-1}$, we obtain that $\gamma^{-1} \tilde v (+\infty) = \infty$ and $\gamma^{-1} \tilde v (-\infty)$ is also on the interval $]\gamma^-,0[$. As a consequence, $\gamma^{-1} \tilde v(t'_\gamma)$, which is the intersection of the vertical ray $\gamma^{-1}\tilde{v}(\mathbb{R}^+)$ and the geodesic $(\gamma^{-}, \gamma^{+})$, is at left of the point $\tilde{u}(t_\gamma)$.
 	Therefore, $\tilde v(t'_\gamma)\in(\gamma^{-},\gamma^{+})$ has to be on the geodesic segment
 	$[\tilde{u}(t_\gamma), \gamma \tilde{u}(t_\gamma)]$. 
 	
 	We conclude that $d(\tilde v(t'_\gamma),\tilde{u}(t_\gamma) )$ and $d(v(t'_\gamma), \gamma \tilde{u}(t_\gamma) )$ are both less than $\ell(\gamma)$. Moreover, by a triangular inequality,
 	\[
 	|t_\gamma - t'_\gamma | \le d(\tilde v(t'_\gamma),\tilde{u}(t_\gamma) ) \le \ell(\gamma).
 	\]

 	We first prove Proposition \ref{key_prop} on $S$, replacing $d_1$ by $d$.	
 	
 	
 	\noindent\textbf{Case 1.}  If $0\le t\le t_\gamma$, then 
 	\[
 	d\bigl(\tilde v(t),\tilde u(t)\bigr)\le 2\ell(\gamma).
 	\] 

 Since the distance in $\mathbb{H}^2$ between two geodesic rays starting at the same point is increasing, it follows 
 \[
  	\forall\,0\le t\le t_\gamma,\qquad
  	d\bigl(\tilde v(t),\tilde u(t)\bigr)\le d\bigl(\tilde v(t_\gamma),\tilde u(t_\gamma)\bigr).
  	\]
Using the bounds established above, one obtains
 	\[
 	d\bigl(\tilde v(t_\gamma),\tilde u(t_\gamma)\bigr) \le 
 	|t_\gamma - t'_\gamma | + d\bigl(\tilde v(t'_\gamma),\tilde u(t_\gamma)\bigr)  \le
 	2\ell(\gamma),
  	\]
which implies the statement.
 	
%
 	
 	\smallskip
 	\noindent\textbf{Case 2:} If $ t\ge  t_\gamma$, then 
 	\[
 	d\bigl(\tilde v(t), \gamma \tilde u(t)\bigr)\le 2\ell(\gamma).
 	\] 
 	
 	Put
 	\[
 	\gamma^{-1}\tilde v(t)=a+i\,e^{\tau+t}\quad\text{and}\quad \tilde u(t)=i\,e^{t}.
 	\]
 	We have \cite[\S7.20]{Beardon} 
 	\begin{equation*}
 		\left(\sinh\frac{d\bigl(\gamma^{-1}\tilde v(t),\tilde u(t)\bigr)}{2}\right)^{2}
 		=
 		\frac{a^{2}+e^{2t}(e^{\tau}-1)^{2}}{4e^{\tau}e^{2t}}
 		=
 		\frac{a^{2}}{4e^{\tau}}\,e^{-2t}
 		+\frac{(e^{\tau}-1)^{2}}{4e^{\tau}}.
 	\end{equation*}
 	It follows that $d\bigl(\gamma^{-1}\tilde v(t),\tilde u(t)\bigr)$ is decreasing, and hence for $t\ge t_\gamma$,
 	\[
 	d\bigl(\gamma^{-1}\tilde v(t),\tilde u(t)\bigr)
 	\le d\bigl(\gamma^{-1}\tilde v(t_\gamma),\tilde u(t_\gamma)\bigr).
 	\]
 	Moreover,
 	\[
 	d\bigl(\gamma^{-1}\tilde v(t_\gamma),\tilde u(t_\gamma)\bigr)
 	\le |t_\gamma-t'_\gamma| + d\bigl(\gamma^{-1}\tilde v(t'_\gamma), \tilde u(t_\gamma) \bigr).
 	\]
 	By the bounds proved above, each term on the right hand side is bounded by $\ell (\gamma)$.
 	We conclude that, for $t\ge t_\gamma$, $d\bigl(\tilde v(t), \gamma \tilde u(t)\bigr) = d\bigl( \gamma^{-1} \tilde v(t), \tilde u(t)\bigr)  \le 2\ell(\gamma)$. 
 	
 	\smallskip
 	
 Let us now prove Proposition \ref{key_prop}. Recall that 
 \[
d_1\!\left(g_{t} \tilde v ,\, g_t\tilde u\right) = d(\tilde v (t), \tilde u (t)) + d(\tilde v (t+1), \tilde u (t+1)). 
 \]
 We distinguish three scenarios:
 
 \begin{enumerate}[(a)]
 	\item If $t \le t_\gamma -1 $, by the first case, we have that both
 	$d(\tilde v(t), \tilde u(t))$ and \\
 	$d(\tilde v(t+1), \tilde u(t+1))$
 	are bounded by $2\ell(\gamma)$. Then
 	$
 	d_1(g_{t}\tilde v, g_t \tilde u) \le 4\ell(\gamma).
 	$
 	
 	\item If $t \ge t_\gamma$, by the second case, we can still
 	bound both
 	$d(\tilde v(t), \gamma \tilde u(t))
 	$ and $ 
 	d(\tilde v(t+1), \gamma \tilde u(t+1))$
 	by $2\ell(\gamma)$. Then
 	$
 	d_1(g_{t}\tilde v, \gamma g_t \tilde u) \le 4 \ell(\gamma).
 	$
 	
 	\item If $t_\gamma -1 < t < t_\gamma$. We have
 	\begin{align*}
 		d_1(g_{t}\tilde v, g_t \tilde u) &= d(\tilde v (t), \tilde u (t)) + d(\tilde v (t+1), \tilde u (t+1)) \\
 		&\le  d(\tilde v (t), \tilde u (t)) + d(\tilde v (t+1), \gamma \tilde u (t+1)) + d( \gamma \tilde u (t+1), \tilde u (t+1)).
 	\end{align*}
	The first and the second terms in the last line are bounded by $2\ell (\gamma)$ by Cases 1 and 2 above, respectively. 
 	
	In order to bound the third term we apply the formula of the displacement function \cite[Theorem 7.35.1]{Beardon}. Writing $z= \tilde u(t+1)$, we obtain
 	\begin{equation*}\label{eq:5.11}
 		\sinh\!\left(\frac{d(z, \gamma z)}{2}\right)
 		=
 		\cosh \bigl(d(z, (\gamma^-, \gamma^+))\bigr)
 		\sinh\!\left(\frac{\ell(\gamma)}{2}\right).
 	\end{equation*}
 
 We notice that \[ d(z, (\gamma^-, \gamma^+) )  \le d(\tilde u(t+1), \tilde u (t_\gamma) ) = t+1 - t_\gamma \le 1 .\] We obtain 
 \[
 \sinh\!\left(\frac{d(z, \gamma z)}{2}\right) \leq 2 \sinh\!\left(\frac{\ell(\gamma)}{2}\right).
 \]
 Since $2\sinh(\frac{\ell(\gamma)}{2})\leq \sinh \ell(\gamma)$, then the third term $d(z,\gamma z)$ is bounded by $2\ell(\gamma)$. We obtain that, if $t_\gamma -1 < t < t_\gamma$, then 
 \[
 d_1(g_{t}\tilde v, g_t \tilde u) \le 6 \ell(\gamma).
 \]

 \end{enumerate}

Finally, we observe that 
\begin{align*}
	d_1\!\left(g_{t+\tau} \tilde v ,\, g_t\tilde u\right) & \le d_1\!\left(g_{t+\tau} \tilde v ,\, g_t\tilde v\right) + d_1\!\left(g_{t} \tilde v ,\, g_t\tilde u\right)\\
	& =  2|\tau| + d_1\!\left(g_{t} \tilde v ,\, g_t\tilde u\right) \\
	&\le  2\ell(\gamma) + d_1\!\left(g_{t} \tilde v ,\, g_t\tilde u\right),
\end{align*}
and similarly we have
\[
d_1\!\left(g_{t+\tau} \tilde v ,\, \gamma g_t\tilde u\right) \le 2\ell(\gamma) + d_1\!\left(g_{t} \tilde v ,\, \gamma g_t\tilde u\right).
\]
Together with the previous bounds, these imply the statement.
 \end{proof}


\section{Proof of Main Theorem}
\begin{proposition}\label{prop3}
	Let $S=\Gamma\backslash\H^2$ be a hyperbolic surface and $u\in T^1S$. If $u(\R^+)$ crosses infinitely many closed geodesics with length converging to $0$, then there exists a sequence $(\gamma_n')_{n\geq 0}$ of hyperbolic isometries in $\Gamma$, and an isometry $f$ such that:
	\begin{enumerate}
		\item $(\gamma_n=f\gamma_n'f^{-1})_{n\geq 0}$ is a nested sequence (Definition \ref{nested_sequence}),
		\item The geodesic ray $f^{-1}(\tilde u_0)(\R^+)\subset \H^2$ projects onto $u(\R^+)\subset S=\Gamma\backslash\H^2$.
	\end{enumerate}
\end{proposition}

\begin{proof}
	Let $\tilde u\in T^1\mathbb H^2$ be a lift of $u$.	
	Since $\mathrm{PSL}(2,\mathbb R)$ acts transitively on $T^1\mathbb H^2$, there exists an isometry
	$
	f\in \mathrm{PSL}(2,\mathbb R)
	$
	such that
	$
	\tilde u=f^{-1}(\tilde u_0).
	$
	Therefore the geodesic ray
	$
	f^{-1}(\tilde u_0)(\mathbb R^+)=\tilde u(\mathbb R^+)
	$
	projects onto
	$
	u(\mathbb R^+)\subset S.
	$
	
	By assumption, the ray $u(\mathbb R^+)$ crosses infinitely many closed geodesics whose lengths converge to $0$. For each such closed geodesic, choose a hyperbolic element
	$
	\gamma_n'\in \Gamma
	$
	whose axis projects onto that closed geodesic and meets the ray $\tilde u(\R^+)$. Passing to a subsequence, we may assume that
	$
	\ell(\gamma_n')
	$
	is decreasing and converges to $0$.
	
	Set
	$
	\gamma_n=f\gamma_n'f^{-1}.
	$
	Clearly, each $\gamma_n$ is hyperbolic and
	$
	\ell(\gamma_n)=\ell(\gamma_n').
	$
	Moreover, since the axis of $\gamma_n'$ meets $\tilde u(\mathbb R^+)$, the axis of $\gamma_n$ meets
	$
	f(\tilde u(\mathbb R^+))=\tilde u_0(\mathbb R^+).
	$
	
	We shall now prove that, after possibly replacing some $\gamma_n$ by their inverses and extracting a subsequence, $(\gamma_n)_{n\geq 0}$ is a nested sequence.
	
	\medskip
	
	\noindent
	Let
	$
	F=\langle \gamma_n \; ;\; n\geq 0\rangle.
	$
	Since
	$
	F\subset f\Gamma f^{-1}
	$
	and $\Gamma$ is discrete, $F$ is discrete. Thus, condition~(1) is verified.
	
	\medskip
	
	\noindent
	Since the axis of each $\gamma_n$ intersects the vertical ray
	$
	\tilde u_0(\mathbb R^+)=[i,\infty),
	$
	its two endpoints lie on opposite sides of $0$. Hence, after possibly replacing $\gamma_n$ by $\gamma_n^{-1}$, we may assume that
	$
	\gamma_n^-<0<\gamma_n^+
	 \text{ for all }n,
	$ and that the sequence $(\gamma_n)_{n\geq 0}$ and $(\gamma_n^-)_{n\geq 0}$ are increasing and decreasing, respectively, which gives conditions~(2) and (3). 
	
	By construction, the axis of each $\gamma_n$ intersects $\tilde u_0(\mathbb R^+)$. Since this axis has endpoints $\gamma_n^-<0<\gamma_n^+$, the intersection consists of exactly one point. Therefore, for every $n\geq 0$, there exists $t_n>0$ such that
	\[
	(\gamma_n^-,\gamma_n^+)\cap \tilde u_0(\mathbb R^+)=\tilde u_0(t_n).
	\]
	This is condition~(4).
	
	\medskip
	
	\noindent
	Finally, since the sequence $(\ell(\gamma_n'))$ is chosen decreasing with limit $0$, and
	$
	\ell(\gamma_n)=\ell(\gamma_n'),
	$
	therefore $(\ell(\gamma_n))_{n\geq 0}$ is decreasing and converges to $0$. This is condition~(5).
	
	\medskip
\end{proof}


Let $S=\Gamma\backslash \H^2$ and $u\in T^1 S$ satisfying the conditions of the Main Theorem \ref{thE} and $f\in\operatorname{PSL(2,\R)}$ given by Proposition \ref{prop3}, Clearly, it is enough to prove the Main Theorem for $S_0=f\Gamma f^{-1}\backslash \H^2$ and $u_0\in T^1S_0$ lifting to $\tilde u_0$.

Let $(\gamma_n')_{n\geq 0}$ given by Proposition \ref{prop3}, and set $\gamma_n=f\gamma_n'f^{-1}$, which is a nested sequence. 
 Our goal is to give conditions on a subsequence $\alpha=(\alpha_n)_{n\geq0}$ of $(\gamma_n)_{n\geq 0}$ to guarantee the existence of $\tilde w_\alpha \in T^1\H^2$ with $\tilde w_\alpha(+\infty)=\lim\limits_{n\to +\infty}\alpha_0\ldots\alpha_n(\infty)$, such that its projection $w_\alpha$ onto $T^1S_0$ satisfies $ w_\alpha\in W^{ss}u_0-h_\R u_0$.




\subsection{Definition of $v_n$ and $w_\alpha$}

\par\,

	Let $(\alpha_n)_{n\in\\N}$ be a subsequence of $(\gamma_n)_{n\in \N}.$ We introduce $\beta_n=\alpha_0\ldots\alpha_n\in\Gamma_0,$ and $\tilde v_n\in T^1\H^2$ defined by:
	\begin{itemize}
		\item $\tilde v_n(0)=i,$
		\item $\tilde v_n(+\infty)=\beta_n(\infty)$.
	\end{itemize}
	

\begin{lemma}\label{lemma_vectors}
	The sequence $(\tilde v_n)_{n\in\N}$ converges towards $\tilde v_\alpha\in T^1 \H^2 $ defined by $\tilde v_\alpha (0)=i$ and $\tilde v_\alpha(+\infty)=\lim\limits_{n\to +\infty}\beta_n(\infty)$.
	
\end{lemma}

\begin{proof}
	We have to prove that the sequence $(\beta_n(\infty))_{n\in\N}\subset \partial \H^2$ converges. 
	Using the dynamics of $\gamma_n$ we have:
	\begin{itemize}
		\item $\forall n\geq 1$, $0<\alpha_{n-1}\alpha_n\infty<\alpha_{n-1}\infty$,
		\item $\forall i\geq 0,$ if $0<x<y$, then $0<\alpha_i x<\alpha_iy.$
	\end{itemize}
	It follows that $(\beta_n(\infty))_{n\in\N}$ is a decreasing sequence of positive numbers and hence converges.

 \end{proof}



Applying $\alpha_0^{-1}$ to $\tilde v_n(0)=i$ maps it into the same connected component of $\H^2-(\alpha_0^-,\alpha_0^+)$, repeating this argument inductively leads to the conclusion that $\alpha_{n}^{-1}\ldots\alpha_0^{-1}(i)$ belongs to the connected component of $\H^2-(\alpha_n^-,\alpha_n^+)$ containing $i$. It follows that
 the ray $\alpha_{n}^{-1}\ldots\alpha_0^{-1}\tilde v_n(\R ^+)=\beta_n^{-1}\tilde v_n(\R ^+)$ crosses $(\alpha_{n+1}^-,\alpha_{n+1}^+) $ and hence that $\tilde v_n(\R^+)$ crosses $(\beta_n(\alpha_{n+1}^-)),\beta_n(\alpha_{n+1}^+) )$.

As a consequence we have for all $n\geq 0$
$$
\tilde v_{n+1}=\mathrm{Wind}_{\beta_n\alpha_{n+1}\beta_n^{-1}}(\tilde v_n),
$$
and accordingly to Proposition \ref{bound_wind}
$$
\left|\tau_{\beta_n\alpha_{n+1}\beta_n^{-1},\tilde v_n}\right |\leq \ell(\alpha_{n+1}).
$$

	\begin{figure}[htbp]
	\centering
	\begin{normalsize} 
		\def\svgwidth{13cm}
		\input{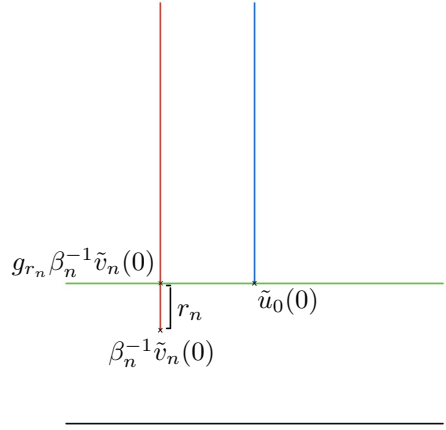}
	\end{normalsize}
	\caption{$g_{r_n}\beta_n^{-1}\tilde v_n\in h_\R\tilde u$.}
	\label{dessin13}
\end{figure}

The next step of the construction is to ensure that the time spent to wind around all the geodesics $(\alpha_n^-,\alpha_n^+)$ is finite.

\begin{proposition}[Convergence of the time to wind around closed geodesics]\label{lema_4}
	
Let $r_n=B_\infty(\beta_n^{-1}i,i))$. If for every $n\geq 0$,  $\ell(\alpha_n)<\frac{1}{2^n}$, then the sequence $(r_n)_{n\in\N}$ converges towards a real number $r_\alpha$.    
\end{proposition}

\begin{proof}
We have $\left | r_{n+1}-r_n  \right |=\left | B_\infty(\alpha_{n+1}^{-1}\beta_n^{-1}(i), \beta_n^{-1}(i)) \right |$.
Hence
 $$\left | r_{n+1}-r_n  \right |=\left | B_{\beta_n \infty}(\beta_n\alpha_{n+1}^{-1}\beta_n^{-1}(i),i) \right |.$$
 It follows that

$$
\left | r_{n+1}-r_n  \right |=\left | \tau_{\beta_n\alpha_{n+1}\beta_n^{-1}}(\tilde v_n) \right|\leq \ell(\alpha_{n+1})\leq \frac{1}{2^{n+1}}.$$
As a consequence $(r_n)_{n\in \N}$ is a Cauchy sequence and hence converges.
\end{proof}

For a subsequence $\alpha$ of $(\gamma_n)_{n\N}$ satisfying the assumptions of Proposition \ref{lema_4} we will define $w_\alpha$ as follows:
$$w_\alpha=g_{r_\alpha}v_\alpha=\lim\limits_{n\to \infty}g_{r_n}v_n.$$

\begin{remark}\label{obs_orbit}
As direct consequence of the definition of $w_\alpha$ we deduce that $w_\alpha\in \overline{h_\R u_0}$.
\end{remark}

\subsection{Construction of suitable $\alpha=(\alpha_n)_{n\in\N }$}

\begin{proposition}\label{lema_clave}
	Let $\alpha=(\alpha_n=\gamma_{p_n})_{n\geq0}$ be a subsequence of $(\gamma_n)_{n\geq0}$ verifying $\ell({\alpha_n})\leq \frac{1}{2^{2n+3}}$. Then, 
	for every $n,\ell\geq 0$ there exists $T_\ell\ge 0$ with the property that,
	\[\forall \,t\ge T_\ell,~
	d_1\bigl(g_t g_{r_n}v_n,\, g_t u_0\bigr)
	\le \left(\sum_{k=0}^{n}\frac{1}{2^{k}}\right)\frac{1}{2^{\ell}} .
	\]
\end{proposition}

%

\subsubsection*{Proof of Proposition~\ref{lema_clave}}
For $m\ge 0$, we say that $\alpha_0,\dots,\alpha_m$ satisfy Property $(P_m)$ if:
\begin{itemize}
	\item for every $n,\ell\in\{0,\dots,m\}$, there exists $T_l\ge 0$ such that for every $t\ge T_l$,
	\[ \forall t\ge T_l,\quad
	d_1\bigl(g_t g_{r_n} v_n,\, g_t u_0\bigr)
	\le \left(\sum_{k=0}^{n}\frac{1}{2^{k}}\right)\frac{1}{2^{\ell}} .
	\]
\end{itemize}

\medskip
\noindent\textbf{Step $(P_0)$.}
By definition, $\tilde v_0=\operatorname{Wind}_{\alpha_0}(\tilde u_0)$ and $r_0=\tau_{\alpha_0,\tilde v_0}$.
Recall that $\alpha_0=\gamma_{p_0}$ is such that $\ell(\alpha_0)\le \frac{1}{2^3}$.
Applying Key Proposition \ref{key_prop} one obtains $(P_0)$:
\begin{itemize}
	\item$\forall t\ge T_0=0$,
	\[
	d_1\bigl(g_tg_{r_0} v_0,\, g_t u_0\bigr)\le 1.
	\]
\end{itemize}

\medskip
\noindent\textbf{Induction step.}
Let $m\geq0$, and suppose that $(P_m)$ is satisfied.
Since $|r_0|\le \ell(\alpha_0)\le 1$ (Proposition \ref{bound_wind}) and
$|r_n-r_{n-1}|\le \ell(\alpha_n)\le \frac{1}{2^{n}}$ (Proof of Proposition \ref{lema_4}), we have
\[
\forall n\in\{0,\dots,m\},\qquad
|r_n|\le 1+\sum_{k=0}^{n}\frac{1}{2^{k}}\le 3.
\]

\medskip
Recall that $\alpha_{m+1}=\gamma_{p_{m+1}}$ is such that $\ell(\alpha_{m+1})\le \frac{1}{2^{2(m+1)+3}}$. We have to prove the existence of	
$T_{m+1}\ge 0$ satisfying:
\begin{enumerate}
	\item for every $n\in\{0,\dots,m\},$ 
	\[\forall \, ~ t\ge T_{m+1},~
	d_1\bigl(g_t g_{r_{n}}v_n,\, g_tu_0\bigr)
	\le \left(\sum_{k=0}^{m}\frac{1}{2^{k}}\right)\frac{1}{2^{m+1}};
	\]
	\item for every $\ell\in\{0,\dots,m+1\},$ 
	\[\forall \, ~ t\ge T_{\ell},~
	d_1\bigl(g_t g_{r_{m+1}}v_{m+1},\, g_tu_0\bigr)
	\le \left(\sum_{k=0}^{m+1}\frac{1}{2^{k}}\right)\frac{1}{2^{\ell}} .
	\]
\end{enumerate}

\medskip
\noindent\textbf{Case (1).}
By construction, there exists $s_n\in \R$ such that $\beta_n^{-1}g_{r_n}\tilde v_n= h_{s_n}(\tilde u_0)$.
Since $g_th_sg_{-t}=h_{se^{-t}}$, we have that
$$
d_1(g_th_{s_n}\tilde u_0,g_t\tilde u_0)=d_1(h_{s_ne^{-t}}g_t\tilde u_0,g_t\tilde u_0).
$$
Since $d((h_{s_ne^{-t}}g_t\tilde u_0)(0),\tilde u(t))\leq |s_n|e^{-t}$, we obtain
$$d_1(g_th_{s_n}\tilde u_0 ,g_t \tilde u_0)\leq |s_n|(e^{-t}+e^{-(t+1)}).$$

Let $t_n\geq 0$ such that
$$
\forall\, t\geq t_n, ~d_1(g_th_{s_n}\tilde u_0,g_t\tilde u_0)\leq \frac{1}{2^{m+1}}.
$$
Take $T_{m+1}'=\max\{t_0,\dots,t_m\}$: 	
\[\forall \, t\ge T_{m+1}', ~
d_1\bigl(g_t\beta_n^{-1}g_{r_n}\tilde v_n,\, g_t\tilde u_0\bigr)
\le \frac{1}{2^{m+1}}.
\]
This implies
\begin{equation}\label{case1}
	\forall \, t\ge T_{m+1}', ~
	d_1\bigl(g_t g_{r_n} v_n,\, g_t u_0\bigr)
	\le \left(\sum_{k=0}^{m}\frac{1}{2^{k}}\right)\frac{1}{2^{m+1}}.
\end{equation}

\medskip
\noindent\textbf{Case (2).}
We have
\[
d_1\bigl(g_t g_{r_{m+1}} v_{m+1},\, g_t u_0\bigr)
\le
\underbrace{d_1\bigl(g_t g_{r_{m}} v_{m},\, g_t u_0\bigr)}_{(a)}
+
\underbrace{d_1\bigl(g_t g_{r_{m+1}} v_{m+1},\, g_t g_{r_{m}} v_{m}\bigr)}_{(b)}.
\]

\medskip
\noindent\textbf{Part (a).}
Since $(P_m)$ is satisfied, for any $\ell\in\{0,\dots,m\}$ there exists $T_\ell\ge 0$ such that for every $t\ge T_\ell$,
\[
(a)\le \left(\sum_{k=0}^{m}\frac{1}{2^{k}}\right)\frac{1}{2^{\ell}} .
\]
Moreover, by (\ref{case1}), we have
\[ \forall\, t\ge T_{m+1}^{'}:~
(a)\le \left(\sum_{k=0}^{m}\frac{1}{2^{k}}\right)\frac{1}{2^{m+1}}.
\]

\medskip
\noindent\textbf{Part (b).}
Recall that $\tilde v_{m+1}=\operatorname{Wind}_{\beta_m \alpha_{m+1}\beta_m^{-1}}(\tilde v_m)$ and
\[
\tau_{\beta_m \alpha_{m+1}\beta_m^{-1},\,\tilde v_m}=r_{m+1}-r_m.
\]
Since $\ell(\alpha_{m+1})\leq \frac{1}{2^{2(m+1)+3}}$, applying Key Proposition \ref{key_prop}, we obtain:
\[\forall \, t\geq 0,~
d_1\bigl(g_{t+r_{m+1}-r_m}\tilde v_{m+1},\, g_t\tilde v_{m}\bigr)
\le \frac{1}{2^{2(m+1)}}.
\]
Hence, for every $s>-r_m$,
\[
d_1\bigl(g_s g_{r_{m+1}} v_{m+1},\, g_s g_{r_m} v_m\bigr)
\le \frac{1}{2^{2(m+1)}}.
\]
Since $|r_m|\le 3$, one has for every $t\ge 3$,
\[
d_1\bigl(g_t g_{r_{m+1}} v_{m+1},\, g_t g_{r_m} v_m\bigr)
\le \frac{1}{2^{2(m+1)}}.
\]

\medskip
Finally, let
\[
T_{m+1}=\max\{3,\,T_{m+1}'\}.
\]
Since for every $\ell\in\{0,\dots,m+1\}$,
\[
\frac{1}{2^{2(m+1)}}\le \frac{1}{2^{m+1}}\cdot\frac{1}{2^{\ell}},
\]
one obtains for every $\ell\in\{0,\dots,m+1\}$ and every $t\ge T_\ell$,
\[
d_1\bigl(g_t g_{r_{m+1}} v_{m+1},\, g_t u_0\bigr)
\le \left(\sum_{k=0}^{m+1}\frac{1}{2^{k}}\right)\frac{1}{2^{\ell}}.
\]
\qed

\begin{remark}
Since the only condition imposed on $\alpha_{m+1}$ at the Induction Step is \\ $\ell(\alpha_{m+1})\leq \frac{1}{2^{2(m+1)+3}}$ and since the sequence $(\ell(\gamma_n))_{n\geq 0}$ is decreasing  and converges to $0$, it follows that the set of such sequences $\Sigma$ is uncountable.  

\end{remark}

\subsection{End of the Proof of the Main Theorem}\label{end_proof}

\begin{proof}
		We denote $\Sigma$ the set of $(\alpha_n)_{n\geq 0}$ with $\ell(\alpha_n)\leq \frac{1}{2^{(2n+3)}}$.
		Let $(\alpha_n)_{n\in\N}$ a subsequence in $\Sigma$ (see Proposition \ref{lema_clave}). Referring back to the notations of Lemma \ref{lemma_vectors} and Proposition \ref{lema_4}, we will prove that the vector $w_\alpha=g_{r_\alpha} v_\alpha$ of $T^1S_0$ belongs to $W^{ss}u_0.$ 
		Observe that, thanks to Proposition \ref{lema_clave}, for all $n,l\geq 0$ and $t\geq T_l:$
		$$ d_1(g_{t+r_n}v_n,g_t u)<\frac{1}{2^{l-1}}$$ And thus, taking the limit on $n$, we can state that for all $t\geq T_l$:
		$$ d_1(g_{t+r}v_\alpha,g_t u_0)<\frac{1}{2^{l-1}}, $$
		then, for all $\varepsilon >0 $ there exists $l\geq 0$ such that for all $t\geq T_l$:
		$$ d_1(g_tw_\alpha,g_tu_0)=d_1(g_{t+r}v_\alpha,g_tu_0)\leq\varepsilon,$$
		which implies that $w\in W^{ss}u$.

Let
$
W(\infty)=\left\{\tilde w_{\alpha}(+\infty),\ \alpha\in\Sigma\right\}.
$
This set is uncountable. Suppose it is not, then there exists a sequence
$
\bigl(\tilde w_{\alpha^i}(+\infty)=x_i\bigr)_{i\in\mathbb N}
$
with $\alpha^i\in\Sigma$ such that
$
W(\infty)=\{x_i,\ i\in\mathbb N\}.
$
Let us construct $\alpha'\in\Sigma$ such that
$
\tilde w_{\alpha'}(+\infty)\neq x_i
$
for any $i\in\mathbb N$.

Choose
$
\alpha_0'=\gamma_{p_{k_0}}
$
such that:
\begin{itemize}
	\item $\ell(\alpha_0')\leq \dfrac{1}{2^3}$,
	\item $
	\alpha_0^{\prime +}>x_0.
$
\end{itemize}

By induction, if $\alpha_0',\ldots,\alpha_n'$ are chosen, take
$
\alpha_{n+1}'
$
satisfying:
\begin{itemize}
	\item $\ell(\alpha_{n+1}')\leq \dfrac{1}{2^{2(n+1)+3}}$,
	\item $
	\alpha_{n+1}^{\prime +}>
	\alpha_n^{\prime -1}\cdots \alpha_0^{\prime -1}(x_{n+1}).
	$
\end{itemize}

This construction is possible since
$
\lim\limits_{n\to+\infty}\gamma_n^+=\infty
$
and the sequence
$
\bigl(\ell(\gamma_n)\bigr)_{n\geq 0}
$
decreases to $0$.

Observe that, using the position of the axis
$
(\gamma_n^-,\gamma_n^+) \text{ for } n\geq 0
$
and the dynamics of $\gamma_n$, for any $0\leq n\leq m$, we have:
\[
\alpha_n'\ \alpha_{n+1}'\ \ldots\ \alpha_m'(\infty)>\alpha_n^{\prime +}.
\]

We deduce for $n\geq 1$:
\[
\alpha_{n-1}^{\prime -1}\cdots \alpha_0^{\prime -1}
\bigl(\tilde w_{\alpha'}(+\infty)\bigr)
>
\alpha_n^{\prime +}
\qquad (*)
\]

By construction, for $n\geq 1$, we have
\[
\alpha_n^{\prime +}>
\alpha_{n-1}^{\prime -1}\cdots \alpha_0^{\prime -1}(x_n)
\qquad (**)
\]

Both inequalities imply:
\[
\forall n\geq 1,\qquad
\alpha_{n-1}^{\prime -1}\cdots \alpha_0^{\prime -1}
\bigl(\tilde w_{\alpha'}(\infty)\bigr)
\neq
\alpha_{n-1}^{\prime -1}\cdots \alpha_0^{\prime -1}(x_n),
\]
and hence
$
\tilde w_{\alpha'}(+\infty)\neq x_n.
$
Moreover
$
\tilde w_{\alpha'}(+\infty)\neq x_0
$
since
$
\alpha_0^{\prime +}>x_0
\text{ and }
\tilde w_{\alpha'}(+\infty)>\alpha_0^{\prime +}.
$

We conclude that $W(\infty)$ is not countable.

In particular, since $\Gamma$ is countable, the set
$
\Gamma W(\infty)=\{\gamma(x),\ \gamma\in\Gamma,\ x\in W(\infty)\}
$
is an uncountable disjoint union of $\Gamma$-orbits.

Let $v\in W^{ss}(u_0)$, since $v\in h_\R(u_0)$
implies
$
\tilde v(\infty)\in \Gamma\infty,
$
the set $W^{ss}(v_0)$ is an uncountable union of horocycle trajectories.

\end{proof}

\appendix

\section{Equality between $W_{d_1}^{ss}(u)$ and $W_{d_2}^{ss}(u)$}\label{apendice}

Let $\tilde v,\tilde w\in T^{1}\mathbb H^{2}$ such that $\tilde v(0)\neq \tilde w(0)$.
Consider the geodesic ray starting at $\tilde v(0)$ and passing through $\tilde w(0)$.
Denote $C_{\tilde v,\tilde w}\in T^{1}\mathbb H^{2}$ its unitary tangent vector at $\tilde v(0)$.

By construction, for $s=d(\tilde v(0),\tilde w(0))$, the base point of $g_{s}C_{\tilde v,\tilde w}$ is $\tilde w(0)$.
Denote $\alpha_{\tilde v,\tilde w}\in[0,2\pi)$ (resp.\ $\beta_{\tilde v,\tilde w}$) the oriented angle
defined by $\widehat{(C_{\tilde v,\tilde w},\tilde v)}$ (resp.\ $\widehat{\left(g_{s}(C_{\tilde v,\tilde w}),\tilde w\right)}$).

We introduce a distance
$ 
d_{2}:T^{1}\mathbb H^{2}\times T^{1}\mathbb H^{2}\longrightarrow \mathbb R_{+}
$
defined by
\[
d_{2}(\tilde v,\tilde w)=d\bigl(\tilde v(0),\tilde w(0)\bigr)+\bigl|\alpha_{\tilde v,\tilde w}-\beta_{\tilde v,\tilde w}\bigr|,
\]
where $d$ is the hyperbolic distance.

It turns out that $d_{2}$ is equivalent to the Sasaki distance $d_{Sa}$ induced by the
Sasaki metric on $T^{1}\mathbb H^{2}$ \cite{Sasaki}.

Recall that $d_{1}$ is the distance on $T^{1}\mathbb H^{2}$ defined by
\[
d_{1}(\tilde v,\tilde w)=d\bigl(\tilde v(0),\tilde w(0)\bigr)+d\bigl(\tilde v(1),\tilde w(1)\bigr).
\]

Clearly, both $d_{1}$ and $d_{2}$ are invariant by $G=\mathrm{PSL}(2,\mathbb R)$.

Let $\Gamma$ be a torsion free 
Fuchsian group.
Since $d_{1},d_{2}$ are $G$-invariant, they induce distances, denoted again $d_{1}$ and $d_{2}$,
on $T^{1}S=\Gamma\backslash T^{1}\mathbb H^{2}$ as follows: for $u,v\in T^{1}S$,
\[
\forall i=1,2,\qquad d_{i}(u,v)=\inf_{\gamma\in\Gamma} d_{i}(\tilde u,\gamma\tilde v),
\]
where $\tilde u,\tilde v$ project onto $u,v$.

For $i=1,2$, set
\[
W_{d_{i}}^{ss}(u)=\left\{\,v\in T^{1}S\;\middle|\; \lim\limits_{t\to+\infty} d_{i}(g_{t}u,g_{t}v)=0 \right\},
\]
where $g_{\R}$ is the geodesic flow.

\begin{theorem}\label{equiv_stable_sets}
	\[
	W_{d_{1}}^{ss}(u)=W_{d_{2}}^{ss}(u).
	\]
\end{theorem}

Before proving this theorem, let us prove the following proposition.

\begin{proposition}\label{prop4}
	Let $\tilde u_{0}\in T^{1}\mathbb H^{2}$ such that
	\[
	\tilde u_{0}(0)=i
	\qquad\text{and}\qquad
	\tilde u_{0}(+\infty)=\infty.
	\]
	Let $(\tilde v_{t})_{t\geq 0}$ be a family of vectors in $T^{1}\mathbb H^{2}$.
	Then
	\[
	\lim\limits_{t\to+\infty} d_{1}(\tilde u_{0},\tilde v_{t})=0
	\qquad\Longleftrightarrow\qquad
	\lim\limits_{t\to+\infty} d_{2}(\tilde u_{0},\tilde v_{t})=0.
	\]
\end{proposition}

\begin{proof}

For $t\geq 0$, set
\begin{itemize}
	
	\item$ \alpha_{t}=\alpha_{\tilde u_{0},\tilde v_{t}},
	\qquad
	\beta_{t}=\beta_{\tilde u_{0},\tilde v_{t}},
	$
	\item $\xi_{t}=C_{\tilde u_{0},\tilde v_{t}}(+\infty),
	\qquad
	\xi_{t}'=\tilde v_{t}(+\infty).
	$
\end{itemize}
\medskip

\noindent
$\left( \Longrightarrow \right)$ Suppose $
\lim\limits_{t\to+\infty} d_{1}(\tilde u_{0},\tilde v_{t})=0.
$

Since $
\lim\limits_{t\to+\infty}\tilde v_{t}(0)=i
\text{ and }
\lim\limits_{t\to+\infty}\tilde v_{t}(1)=e\,i,
$
we have $
\lim\limits_{t\to+\infty}\xi_{t}'=\infty.
$

Suppose by contradiction that there exist $\varepsilon>0$ and a sequence $(t_{n})_{n\geq 0}\subset\mathbb R_{+}$
converging to $+\infty$ such that
\begin{equation}\label{eq1}
	\forall n\geq 0,\qquad |\alpha_{t_{n}}-\beta_{t_{n}}|\geq \varepsilon.
\end{equation}

Up to a subsequence, we can suppose $
\lim\limits_{n\to+\infty}\xi_{t_{n}}=\xi.
$

Let $\tilde w\in T^{1}\mathbb H^{2}$ such that
$
\tilde w(0)=i
\text{ and }
\tilde w(+\infty)=\xi,
$

we have
$
\lim\limits_{n\to+\infty} C_{\tilde u_{0},\tilde v_{t_{n}}}=\tilde w,
\lim\limits_{n\to+\infty}g_{s_n}C_{\tilde u_0,\tilde v_{t_n}}=\tilde w, \text{ with } s_n=d(i,\tilde v_{t_n}(0))
\text{ and }
\lim\limits_{n\to+\infty}\tilde v_{t_{n}}=\tilde u_{0}.
$

It follows that
\[
\lim\limits_{n\to+\infty}\alpha_{t_{n}}=\widehat{(\tilde w,\tilde u_{0})}
\qquad\text{and}\qquad
\lim\limits_{n\to+\infty}\beta_{t_{n}}=\widehat{(\tilde w,\tilde u_{0})},
\]
hence
$
\lim\limits_{n\to+\infty}|\alpha_{t_{n}}-\beta_{t_{n}}|=0,
$
which contradicts (\ref{eq1}).

\medskip

\noindent
$\left( \Longleftarrow \right)$ Suppose
$
\lim\limits_{t\to+\infty} d_{2}(\tilde u_{0},\tilde v_{t})=0.
$

Suppose by contradiction that there exist $\varepsilon>0$ and $(t_{n})_{n\geq 0}\subset\mathbb R_{+}$
converging to $+\infty$ such that
\begin{equation}\label{eq2}
	\forall t_{n},\qquad d\bigl(\tilde u_{0}(1),\tilde v_{t_{n}}(1)\bigr)\geq \varepsilon.
\end{equation}

We can suppose
$
\lim\limits_{n\to+\infty}\xi_{t_{n}}=\xi
\text{ and }
\lim\limits_{n\to+\infty}\xi_{t_{n}}'=\xi'.
$

Let $\tilde w\in T^{1}\mathbb H^{2}$ with
$
\tilde w(0)=i
\text{ and }
\tilde w(+\infty)=\xi,
$
and $\tilde w'\in T^{1}\mathbb H^{2}$ with
$
\tilde w'(0)=i
\text{ and } 
\tilde w'(+\infty)=\xi'.
$

Clearly,
$
\lim\limits_{n\to+\infty}\tilde v_{t_{n}}=\tilde w'
\text{ and } 
\lim\limits_{n\to+\infty} C_{\tilde u_{0},\tilde v_{t_{n}}}=\tilde w.
$

It follows that
$
\lim\limits_{n\to+\infty}\alpha_{t_{n}}=\widehat{(\tilde w,\tilde u_{0})}
\text{ and } 
\lim\limits_{n\to+\infty}\beta_{t_{n}}=\widehat{(\tilde w,\tilde w')}.
$ 

Since
$
\lim\limits_{n\to+\infty}|\alpha_{t_{n}}-\beta_{t_{n}}|=0,
$
we obtain
$
\tilde w'=\tilde u_{0}.
$

We deduce that the geodesic segment
$
[\tilde v_{t_{n}}(0),\tilde v_{t_{n}}(1)]
$
converges to the vertical arc
$
[i,e\,i],
$
which contradicts (\ref{eq2}).
\end{proof}

\subsection*{Proof of the Theorem \ref{equiv_stable_sets}}

Let $\tilde u,\tilde v\in T^{1}\mathbb H^{2}$ projecting onto $u,v$ in $T^{1}S$ such that
$v\in W_{d_{i}}^{ss}(u)$.

By definition, there exists $(\gamma_{t})_{t\geq 0}$ in $\Gamma$ such that
\[
\lim\limits_{t\to+\infty} d_{i}\bigl(g_{t}\tilde u,\gamma_{t}g_{t}\tilde v\bigr)=0.
\]

Since $G$ acts on $T^{1}\mathbb H^{2}$ transitively, there exists a sequence of isometries $(f_{t})_{t\geq 0}$ in $G$ such that
\[
g_{t}\tilde u=f_{t}\tilde u_{0}.
\]

Since $G$ acts by isometries on $(T^{1}\mathbb H^{2},d_{i})$, we have
\[
d_{i}\bigl(g_{t}\tilde u,\gamma_{t}g_{t}\tilde v\bigr)
=
d_{i}\bigl(\tilde u_{0},f_{t}^{-1}\gamma_{t}g_{t}\tilde v\bigr).
\]

Applying Proposition \ref{prop4}, we obtain that $v\in W_{d_{j}}^{ss}(u)$, with
$
j\in\{1,2\}.
$
\qed

\end{document}